\documentclass[11pt,twoside, leqno]{article}

\usepackage{amsthm}
\usepackage{amssymb}
\usepackage{amsmath}
\usepackage{mathrsfs}
\usepackage{txfonts}
\usepackage{graphics}
\usepackage{color}
\usepackage[Symbol]{upgreek}

\allowdisplaybreaks 
\newcommand{\ls}{\leqslant}
\newcommand{\gs}{\geqslant}

\newcommand{\R}{\mathbb{R}}

\def\rr{{\mathbb R}}

\def\cn{{\mathbb N}}

\def\loc{{\mathop\mathrm{\,loc\,}}}
\def\Lip{{\mathop\mathrm{\,Lip\,}}}

\def\ez{\epsilon}

\def\gz{{\gamma}}

\def\sz{\sigma}

\def\r{\right}
\def\lf{\left}

\def\la{\langle}
\def\ra{\rangle}

\newtheorem{thm}{Theorem}[section]
\newtheorem{lem}{Lemma}[section]
\newtheorem{prop}{Proposition}[section]
\newtheorem{rem}{Remark}[section]
\newtheorem{cor}{Corollary}[section]
\newtheorem{defn}{Definition}[section]

\numberwithin{equation}{section}

\begin{document}
\arraycolsep=1pt

\title{\bf\Large Heat Kernel Bounds on Metric Measure Spaces\\
and Some Applications
\footnotetext{\hspace{-0.35cm}
2010 \emph{Mathematics Subject Classification} Primary 53C23; Secondary 35K08; 35K05; 42B20; 47B06
\endgraf
{\it Key words and phrases} metric measure space, Ricci curvature, heat kernel, heat equation, Riesz transform
\endgraf
R.J. is partially supported by NSFC (No. 11301029); H.L. is partially supported by NSFCs (No. 11401403 and No. 11371099) and the ARC grant (DP130101302); H.Z. is partially supported by NSFC (No. 11201492) and by Guangdong Natural Science Foundation (No. S2012040007550).
}}
\author{Renjin Jiang, Huaiqian Li \& Huichun Zhang}
\date{}

\maketitle

\begin{center}
\begin{minipage}{11cm}\small
{\noindent{\bf Abstract} Let $(X,d,\mu)$ be a $RCD^\ast(K, N)$ space with $K\in \rr$ and $N\in [1,\infty)$. We derive the upper and lower bounds of the heat kernel on $(X,d,\mu)$ by applying the parabolic Harnack inequality and the comparison principle, and then sharp bounds for its gradient, which are also sharp in time.  For applications, we study the large time behavior of the heat kernel, the stability of solutions to the heat equation, and show the $L^p$ boundedness of (local) Riesz transforms. }\end{minipage}
\end{center}

\section{Introduction and main results}
\hskip\parindent Let $M$ be a complete (smooth) Riemannian manifold with dimension $n\ge 2$, and let $p_t$ be the heat kernel. Denote the length in the tangent space by $|\cdot|$. It is well-known from Li and Yau \cite{ly86} that, if $M$ has nonnegative Ricci curvature,
then for any $\epsilon>0$, there exist positive constants $C(\epsilon)$ and $C_1(\epsilon)$, such that
\begin{equation*}
\frac{1}{C(\epsilon)\mu(B(y,\sqrt t))}\exp\lf(-\frac{d^2(x,y)}{(4-\epsilon)t}\r)
\ls p_t(x,y)\ls \frac{C(\epsilon)}{\mu(B(y,\sqrt t))}\exp\lf(-\frac{d^2(x,y)}{(4+\epsilon)t}\r),
\end{equation*}
and
\begin{equation*}
|\nabla_x p_t(x,y)|\le \frac{C_1(\epsilon)}{\sqrt t\mu(B(y,\sqrt t))}\exp\lf(-\frac{d^2(x,y)}{(4+\epsilon)t}\r),
\end{equation*}
for all $t>0$ and all $x,y\in M$.

It is known that the upper and lower bounds of $p_t(x,y)$ has been extended by Sturm \cite{st2,st3} to the Dirichlet space $(X,\mathcal{E},\mu)$ supporting a (weak) local Poincar\'e inequality and the doubling measure $\mu$, where $X$ is a locally compact separable Hausdorff space, $\mathcal{E}$ is a strongly local, regular symmetric Dirichlet form and $\mu$ is a positive Radon measure. Note that here the distance is the so-called intrinsic distance induced by the Dirichlet form $\mathcal{E}$. For the case of the non-symmetric and time-dependent Dirichlet form in the similar framework, see the recent paper \cite{LierlSaloff-Coste}.

Throughout this work, let $(X,d,\mu)$ be a metric measure space, such that $(X,d)$ is a complete and separable metric space and $\mu$ is a locally finite (i.e., finite on bounded sets) Borel regular measure with support the whole space $X$.

Recently, in the metric measure space $(X,d,\mu)$, Erbar et al.\cite{eks13} and Ambrosio et al. \cite{ams13-1} introduced the Riemannian curvature-dimension condition, denoted by $RCD^*(K,N)$, which is a generalization of the ``Ricci curvature lower bound'' for the non-smooth setting and a strengthening of the curvature-dimension condition introduced by Lott and Villani \cite{lv09} and Sturm \cite{stm4,stm5}. The $RCD^*(K,N)$ space resembles more a Riemannian structure in the sense that Cheeger energy and Sobolev spaces are Hilbertian. We refer the reader to \cite{ags1,ags2,ags3,agmr,ams13,eks13} for more details (see also Section 2 below). In this work, we study the sharp heat kernel bounds on $RCD^\ast(K,N)$ spaces with $K\in \rr$ and $N\in [1,\infty)$, and then give some applications.

Notice that Sturm's results in \cite{stm1,st2,st3} are valid for a metric measure space $(X,d,\mu)$ satisfying the $RCD^\ast(0,N)$ condition with $N\in [1,\infty)$, since in these cases doubling property and Poincar\'e inequality hold; see e.g. \cite{eks13,raj1}. However, for example, the constant in
the exponential term in the Gaussian lower bound in \cite{st3} is not sharp. The first main result below gives a sharper lower bound of $p_t(x,y)$.
The approach is adapted from Strum \cite{st1} in the Riemannian setting, by applying the Laplacian comparison principle established by Gigli in \cite{gi12} and the parabolic Harnack inequality established by Garofalo and Mondino \cite{gm14} and the first named author  \cite{jia14}.

\begin{thm}\label{lowerbound-nonnegative}
Let $(X,d,\mu)$ be a  $RCD^*(0,N)$ space with $N\in [1,\infty)$.
Given any $\epsilon>0$, there exists a positive constant $C_1(\epsilon)$ such that
\begin{equation}\label{kernel}
\frac{1}{C_1(\epsilon)\mu(B(y,\sqrt t))}\exp\Big(-\frac{d^2(x,y)}{(4-\epsilon)t}\Big)\le p_t(x,y)\le \frac{C_1(\epsilon)}{\mu(B(y,\sqrt t))}\exp\Big(-\frac{d^2(x,y)}{(4+\epsilon)t}\Big)
\end{equation}
for all $t>0$ and all $x,y\in X$.
\end{thm}
Combining this, and the Li--Yau inequality from \cite{gm14,jia14}, we can immediately derive a sharp bound for the gradient of the heat kernel. Here, we denote the minimum weak upper gradient of a function $f:X\rightarrow \R$ by $|\nabla f|$ (assume its existence at present).
\begin{cor}\label{gradient-bound-nonnegative}
Let $(X,d,\mu)$ be a  $RCD^*(0,N)$ space with $N\in [1,\infty)$.
Given any $\epsilon>0$, there exists a positive constant $C_1(\epsilon)$ such that
\begin{equation}\label{kernelgradient}
|\nabla p_t(x,\cdot)|(y)\le \frac{C_1(\epsilon)}{\sqrt t\mu(B(y,\sqrt t))}\exp\Big(-\frac{d^2(x,y)}{(4+\epsilon)t}\Big).
\end{equation}
for all $t>0$ and $\mu$-a.e. $x,y\in X$.
\end{cor}

{Notice that in \cite{jia14}, the author used similar forms of \eqref{kernel} and \eqref{kernelgradient}
 but with implicit constants in the exponential terms to obtain the Li-Yau inequality. The strength of Theorem \ref{lowerbound-nonnegative} and Corollary \ref{gradient-bound-nonnegative} is that
 we obtain the sharp constants in the exponential terms.}

We shall then establish the following heat kernel bounds on general $RCD^*(K,N)$ spaces with $K<0$ and $N\in [1,\infty)$.
\begin{thm}\label{lb-negative}
Let $(X,d,\mu)$ be a  $RCD^*(K,N)$ space with $K<0$ and $N\in [1,\infty)$.
Given any $\epsilon>0$, there exist positive constants $C_1(\epsilon),C_2(\epsilon)$, depending also on $K,N$,
  such that
\begin{equation}
\frac{1}{C_1(\epsilon)\mu(B(y,\sqrt t))}\exp\Big(-\frac{d^2(x,y)}{(4-\epsilon)t}-C_2(\epsilon) t\Big)\le p_t(x,y)\le
\frac{C_1(\epsilon)}{\mu(B(y,\sqrt t))}\exp\Big(-\frac{d^2(x,y)}{(4+\epsilon)t}+C_2(\epsilon) t\Big)
\end{equation}
for all $t>0$ and all $x,y\in X$.
\end{thm}

\begin{cor}\label{gradient-bound-negative}
Let $(X,d,\mu)$ be a  $RCD^*(K,N)$ space with $K<0$ and $N\in [1,\infty)$.
Given any $\epsilon>0$, there exist positive constants $C_1(\epsilon),C_2(\ez)$ such that
\begin{equation}\label{kernelgradientnegative}
|\nabla p_t(x,\cdot)|(y) \le \frac{C_1(\epsilon)}{\sqrt t\mu(B(y,\sqrt t))}\exp\Big(-\frac{d^2(x,y)}{(4+\epsilon)t}+C_2(\ez)t\Big).
\end{equation}
for all $t>0$ and $\mu$-a.e. $x,y\in X$.
\end{cor}

{ Notice that in \cite[Lemma 3.3]{mn14}, by using an elementary
argument, Mondino and Naber obtained sharp lower and upper bounds
of heat kernels $p_t(x,y)$ for $x,y\in X$ for $x,y$ satisfying $d(x,y)<10\sqrt t$, and an upper bound for
the quantity $\int_{X\setminus B(x,r)}p_t(x,y)\,d\mu(y)$. }

{In a forthcoming paper, the second named author will use a dimensional free Harnack inequality (cf. \cite{Wang1997,li13}) to investigate heat kernel bounds on $RCD(K,\infty)$ spaces (cf. \cite{agmr,ags1,ags3}).}

The following result generalizes the large time behavior of heat kernels in the Riemannian manifold established
by Li \cite[Theorem 1]{li86} to the present setting with also sharp form.
\begin{thm}\label{large-time}
Let $(X,d,\mu)$ be a $RCD^\ast(0,N)$ space with $N\in \cn$ and $N\ge 2$. Let $x_0\in X$.
If there exists $\theta\in (0,\infty)$ such that $\liminf_{R\to\infty}\frac{\mu(B(x_0,R))}{R^N}=\theta$,
then for any
$x,y\in X$, it holds that
$$\lim_{t\to\infty} \mu(B(x_0,\sqrt t))p_t(x,y)=\omega(N)(4\pi)^{-N/2},$$
where $\omega(N)$ is the volume of the unit ball in $\R^N$.
\end{thm}
{According to Xu \cite{Xu14}, on
an $N$-dimensional Riemannian manifold with non-negative Ricci curvature, if the volume is not maximal growth
(i.e., $\liminf_{R\to\infty}\frac{\mu(B(x_0,R))}{R^N}=0$),
then the limit of $\mu(B(x_0,\sqrt t))p_t(x,y)$ as $t\to\infty$ does not necessarily exist. }


For applications of the bounds on the heat kernel and its gradient, we shall consider stability of solutions
to the heat equation, and the boundedness of (local) Riesz transforms.

The paper is organized as follows. In Section 2,  we give some basic notations
and notions for  Sobolev spaces, differential structures,  curvature-dimension conditions and
heat kernels, and recall some known results. In Section 3, we will provide the proofs of heat kernel and its gradient estimates.

In Section 4, we will prove Theorem \ref{large-time}. Stability of solutions
to the heat equation will be studied there as well. Notice that, since the approach
depends on the comparison results between a $RCD^\ast(0,N)$ space and the Euclidean space $\rr^N$, we can only obtain the result when $N\in \cn$ and $N\ge 2$.
The arguments essentially follows from Li \cite{li86} with some necessary modifications to our non-smooth context, since e.g. there seems no effective  Gauss--Green type formula available so far.

In Section 5, we will establish the boundedness of the Riesz transform $|\nabla (-\Delta)^{1/2}|$ and its local version on $L^p(X)$ for all $p\in (1,\infty)$, where $\Delta$ is the Laplacian (see Definition \ref{glap} below). The approach follows from the known one in Riemannian manifolds. In the smooth setting, since the Riesz transform of smooth functions with compact supports is well defined, and this class of functions is dense in $L^p(X)$ for each $p\in (1,\infty)$, one only needs to deal with smooth functions with compact support. However, in our non-smooth setting, by applying
the results from \cite{cd99,acdh04}, the main issue left should be to find a suitable acting class for the Riesz transform.


We should mention that, all the results we get generalize the known ones in the Riemannian manifold with Ricci curvature bounded below,
and hold in the Alexandrov space with Ricci curvature bounded from below; see \cite{zz10}.

Finally, we make some conventions on notation. Throughout the work, we denote
by $C,c$ positive constants which are independent of the main
parameters, but which may vary from line to line. The symbol
$B(x,R)$ denotes an open ball with center $x$ and radius $R$ with respect to the distance $d$, and $CB(x,R)=B(x,CR).$
The space $\Lip(X)$ denotes the set of all Lipschitz functions
on $X$.

\section{Preliminaries}
\hskip\parindent In this section, we recall some basic notions and several auxiliary results.
\subsection{Sobolev spaces and Differential structures}
\hskip\parindent
Let $(X,d)$ be a complete and separable metric space and let $C([0, 1],X)$ be the space of continuous curves on $[0, 1]$ with values in $X$ equipped with the sup norm. For $t\in [0, 1]$, the map $e_t :C([0, 1],X) \to X$ is the evaluation at
time $t$ defined by
$$e_t(\gz):=\gz_t.$$

A curve $\gz: [0,1] \to X$ is in the absolutely continuous
class $AC^q([0,1],X)$ for some $q\in [1,\infty]$, if there exists $g\in L^q([0,1])$ such that, \begin{eqnarray}\label{ac}
d(\gamma_s,\gamma_t)\leq \int_s^t g(r)\,d r,\quad\mbox{for any }s,t\in [0,1]\mbox{ satisfying } s<t.
\end{eqnarray}
It is true that, if $\gamma\in AC^p([0,1];X)$, then the metric slope
$$\lim_{\delta\rightarrow 0}\frac{d(\gamma_{r+\delta},\gamma_r)}{|\delta|},$$
denoted by $|\dot{\gamma}_r|$, exists for $L^1$-a.e. $r\in [0,1]$, belongs to $L^p([0,1])$, and it is the minimal function $g$ such that \eqref{ac} holds (see Theorem 1.1.2 in \cite{AmbrosioGigliSavare2005}). The length of the absolutely continuous curve $\gamma: [0,1]\rightarrow X$ is defined by $\int_0^1 |\dot{\gamma}_r|\,d r$. We call that $(X,d)$ is a length space if
$$d(x_0,x_1)=\inf\left\{\int_0^1 |\dot{\gamma}_r|\,d r:\, \gamma\in AC^1([0,1],X),\, \gamma_i=x_i,\,i=0,1 \right\},\quad \forall\, x_0,x_1\in X.$$

Let $\mu$ be a locally finite (i.e., finite on bounded sets) Borel regular measure on $(X,d)$ with support the whole space $X$. Throughout the work, we call the triple $(X,d,\mu)$ the metric measure space.

\begin{defn}[Test Plan]
Let $\uppi\in {\mathcal{P}}(C([0, 1],X))$.
We say that $\uppi$ is a test plan if there exists a constant $C>0$ such that
$$(e_t)_\sharp{ \uppi}\le C\mu,\quad\mbox{for all }t\in [0,1],$$
and
$$\int \int_0^1|\dot{\gamma}_t|^2\,dt\,d{\mathcal \uppi}(\gamma)<\infty.$$
  \end{defn}

\begin{defn}[Sobolev Space] \label{sobolev} The Sobolev class $S^2(X)$ (resp. $S_{\mathrm{loc}}^2(X)$)
is the space of all Borel functions $f: X\to \rr$, for which there exists a non-negative function
$G\in L^2(X)$ (resp. $G\in L^2_{\mathrm{loc}}(X)$) such that, for each test plan $\uppi$, it holds
\begin{equation}\label{curve-sobolev}
\int |f(\gz_1)-f(\gz_0)|\,d{\mathcal \uppi}(\gamma)\le \int \int_0^1 G(\gz_t)|\dot{\gamma}_t|\,dt\,d{\mathcal \uppi}(\gamma).
\end{equation}
\end{defn}
It then follows from a compactness argument that,  for each  $f\in S^2(X)$ there exists a unique minimal $G$
in the $\mu$-a.e. sense such that \eqref{curve-sobolev} holds. We then denote the minimal $G$ by $|\nabla  f|$
and call it the minimal weak upper gradient following \cite{ags4}.

The inhomogeneous Sobolev space $W^{1,2}(X)$ is defined as  $S^2(X)\cap L^2(X)$ equipped with the norm
$$\|f\|_{W^{1,2}(X)}:=\lf(\|f\|_{L^2}^2+\| |\nabla  f|\|_{L^2(X)}^2\r)^{1/2}.$$

The local Sobolev space $W^{1,2}_{\mathrm{loc}}(\Omega)$ for an open set $\Omega\subset X$,  and the Sobolev space with compact support
$W^{1,2}_{c}(X)$ can be defined in an obvious manner. The relevant Sobolev spaces have been studied in
\cite{ags4,ch,sh}.

The following terminologies and results are mainly taken from \cite{ags3,gi12}.
\begin{defn}[Infinitesimally Hilbertian Space] Let $(X, d,\mu)$ be a metric measure
space. We say that it is infinitesimally Hilbertian, provided $W^{1,2}(X)$ is a Hilbert space.
\end{defn}
Notice that, from the definition, it follows that $(X, d,\mu)$ is infinitesimally Hilbertian if and only if,
for any $f,g\in S^2(X)$, it holds
$$\||\nabla (f+g)|\|_{L^2(X)}^2+\||\nabla (f-g)|\|_{L^2(X)}^2=2\lf(\||\nabla f|\|_{L^2(X)}^2+\||\nabla g|\|_{L^2(X)}^2\r).$$

\begin{defn} Let $(X, d,\mu)$ be an infinitesimally Hilbertian space, $\Omega\subset X$ an open set and $f, g\in S^2_{\mathrm{loc}}(\Omega)$.
The map $\la \nabla f, \nabla g\ra :\, \Omega \to \rr$ is $\mu$-a.e. defined as
$$\la \nabla f, \nabla g\ra:= \inf_{\ez>0} \frac{|\nabla (g+\ez f)|^2-|\nabla g|^2}{2\ez},$$
with the infimum being intended as $\mu$-essential infimum.
\end{defn}

We shall sometimes write $\la \nabla f, \nabla g\ra$ as $\nabla f\cdot\nabla g$ for convenience.
The inner product $\la \nabla f, \nabla g\ra$ is linear, and satisfies the Cauchy--Schwarz inequality,
the chain rule and the Leibniz rule (see e.g. \cite{gi12}).

With the aid of the inner product, we can define the Laplacian operator as below. Notice that the Laplacian operator is linear
due to $(X, d,\mu)$ being infinitesimally Hilbertian.

\begin{defn}[Laplacian] \label{glap}
Let $(X, d,\mu)$ be an infinitesimally Hilbertian space.
Let $f\in W^{1,2}_{\mathrm{loc}}(X)$. We call $f\in {\mathcal D}_{\mathrm{loc}}(\Delta)$, if there exists $h\in L^1_{\mathrm{loc}}(X)$ such that, for each $\psi\in W^{1,2}_{c}(X)$, it holds
\begin{equation*}
\int_X\la \nabla f,\nabla\psi\ra\,d\mu=-\int_X h\psi\,d\mu.
\end{equation*}
We will write $\Delta f=h$. If $f\in W^{1,2}(X)$ and $h\in L^2(X)$, we then call $f\in {\mathcal D}(\Delta)$.
\end{defn}

From the Leibniz rule, it follows that if $f,g\in {\mathcal D}_{\mathrm{loc}}(\Delta)\cap L^\infty_{\mathrm{loc}}(X)$ (resp. $f,g\in {\mathcal D}(\Delta)\cap L^\infty_{\mathrm{loc}}(X)\cap \Lip(X)$), then $fg\in {\mathcal D}_{\mathrm{loc}}(\Delta)$ (resp. $fg\in {\mathcal D}(\Delta)$) satisfies
$\Delta (fg)=g\Delta f+f\Delta g+2\nabla f\cdot\nabla g$.

\subsection{Curvature-dimension conditions and consequences}
\hskip\parindent Let $(X,d,\mu)$ be  an infinitesimally Hilbertian space. Denote by $H_t$  the heat flow $e^{t\Delta}$ corresponding to
 the Dirichlet form $\big(\mathcal{E},W^{1,2}(X)\big)$, defined by
$$\mathcal{E}(f,g)=\int_X \la \nabla f,\nabla g \ra\, d\mu,\quad f,g\in W^{1,2}(X).$$
From $(X,d,\mu)$ being infinitesimally Hilbertian, it follows
that $H_t$ is linear and the Dirichlet form  $\big(\mathcal{E},W^{1,2}(X)\big)$ is strongly local.

Now we recall the definition of $RCD^\ast(K,N)$ spaces. See e.g. \cite[Sections 3 and 4]{eks13} for other equivalent characterizations. Here and in what follows, for $K=0$, $K/(e^{2Kt}-1):=\lim_{K\rightarrow 0}K/(e^{2Kt}-1)=1/(2t)$.
\begin{defn}\label{rcd}
Let $K\in \rr$ and $N\in [1,\infty)$, and let $(X,d,\mu)$ be a length, infinitesimally Hilbertian space satisfying:
\begin{itemize}
\item[(a)] for some constants $C,c>0$ and some point $o\in X$,
$$\mu(B(o,r))\leq Ce^{cr^2},\quad\mbox{for every }r>0,$$
\item[(b)] each function $f\in W^{1,2}(X)$ with $|\nabla f|\leq 1$ admits a continuous representative,
\item[(c)] for every $f\in W^{1,2}(X)$ and every $t>0$,
\begin{equation}\label{defRCD}
|\nabla H_tf|^2+\frac{4Kt^2}{N(e^{2Kt}-1)}|\Delta H_tf|^2\le e^{-2Kt}H_t(|\nabla f|^2),\quad\mu\mbox{-a.e. in }X.
\end{equation}
\end{itemize}
Then we call $(X,d,\mu)$ a $RCD^\ast(K,N)$ space.
\end{defn}

On the $RCD^\ast(K,N)$ space $(X,d,\mu)$, the measure $\mu$ satisfies the local doubling (global doubling, provided $K\geq 0$) property, which we present in the next lemma (see e.g. \cite[Section 3]{eks13}).
\begin{lem}\label{doubling}
Let $(X,d,\mu)$ be a $RCD^\ast(K,N)$ space with $K\leq 0$ and $N\in [1,\infty)$, and let $x\in X$ and $0<r\leq R<\infty$.
\begin{itemize}
\item[(i)] If $K=0$, then
$$\mu\big(B(x,R)\big)\leq \left(\frac{R}{r}\right)^{N} \mu\big(B(x,r)\big).$$

\item[(ii)] If $K<0$, then
$$\mu\big(B(x,R)\big)\leq \frac{l_{K,N}(R)}{l_{K,N}(r)} \mu\big(B(x,r)\big),$$
where $(0,\infty)\ni t\mapsto l_{K,N}(t)$ is a continuous function depending on $K$ and $N$, and $l_{K,N}(t)=O(e^{tC(K,N)})$ as $t$ tends to $\infty$ for some constant $C(K,N)$ depending on $K$ and $N$.
\end{itemize}
\end{lem}

From the definition of the $RCD^\ast(K,N)$ space, we know that $(X,d)$ is a length space. The (local) doubling property immediately implies that every bounded closed ball in $(X,d)$ is totally bounded. Since $(X,d)$ is also complete, it is then proper and geodesic. Recall that a metric space $(X,d)$ is proper if every bounded closed subset is compact. The properness also implies that the Dirichlet form $(\mathcal{E},W^{1,2}(X))$ is indeed regular.

By \cite[Theorem 3.9]{ags3}, we see that the intrinsic metric induced by the Dirichlet form $(\mathcal{E}, W^{1,2}(X))$, defined as
$$d_{\mathcal{E}}(x,y)=\sup\{\psi(x)-\psi(y):\, \psi\in W^{1,2}(X)\cap C(X),\, |\nabla\psi|\leq 1\, \mu\mbox{-a.e. in }X\},$$
for every $x,y\in X$, coincides with the original one, i.e.,
$$d_{\mathcal{E}}(x,y)=d(x,y),\quad \forall\, x,y\in X.$$
Hence, we can work indifferently with either the distance $d$ or $d_{\mathcal{E}}$.

Recently, Rajala \cite{raj1} proved that an $L^1$ weak local Poincar\'{e} inequality holds on $RCD^\ast(K,N)$ spaces, and hence also an $L^1$ strong local Poincar\'{e} inequality holds by the doubling and geodesic properties and by applying \cite[Theorem 1]{HajlaszKoskela1995}. By using the H\"older inequality, we know that the $L^p$ weak local Poincar\'{e} inequality holds for all $p\in (1,\infty)$.
\begin{lem}\label{localPoicare}
Let $(X,d,\mu)$ be a $RCD^\ast(K,N)$ space with $K\leq 0$ and $N\in [1,\infty)$. Then for every $x\in X$ and every $R>0$, there exists a positive constant $C:=C(K,N,R)$ such that for any $r\in (0,R)$,
\begin{eqnarray}\label{localstrongPoincare}
\int_{B(x,r)} |f-f_B|^2\, d\mu\leq Cr^2\int_{B(x,r)}|\nabla f|^2\, d\mu,\quad\mbox{for all }f\in W^{1,2}(X),
\end{eqnarray}
where $f_B=\frac{1}{\mu(B(x,r))}\int_{B(x,r)} f\, d\mu$. In particular, if $K=0$, then \eqref{localstrongPoincare} holds with constant $C:=C(K,N)$ independent of $R$.
\end{lem}

Now we can apply the results from Sturm \cite[Proposition 2.3]{st2} to immediately deduce that there exists a heat kernel, i.e., a measurable map $(0,\infty)\times X\times X \ni (t,x,y)\mapsto p_t(x,y)\in [0,\infty)$ such that, for any $ t>0$, $f\in L^1(X)+L^\infty(X)$ and  each  $x\in X$,
$$H_tf(x)=\int_X f(y)p_t(x,y)\,d\mu(y);$$
for all $s,t>0$ and  all $x,y\in X$,
$$p_{t+s}(x,y)=\int_X p_t(x,z)p_s(z,y)\,d\mu(z);$$
the function $u:(t,y)\mapsto p_t(x,y)$ is a solution of the equation $\Delta u=\frac{\partial}{\partial t}u$ on $(0,\infty)\times X$ in the weak sense. By the symmetry of the semi-group, $p_t$ is also symmetric, i.e., for every $t>0$, $p_t(x,y)=p_t(y,x)$ for all $(x,y)\in X\times X$. {The doubling property and the local $L^2$ Poincar\'{e} inequality imply that the function $x\mapsto p_t(x,y)$ is H\"{o}lder continuous for every $(t,y)\in (0,\infty)\times X$, by a standard argument; see e.g. \cite[Section 3]{st3}.} Moreover, $H_t$ is stochastically complete (see e.g. \cite[Theorem 4]{stm1}), i.e.,
$$\int_Xp_t(x,y)\, d\mu(y)=1,\quad{\forall }t>0\mbox{ and }\ {\forall }x\in X.$$

\section{Heat kernel bounds}
\hskip\parindent In this section, we shall prove the main results on the heat kernel bounds. We shall  follow the approach of Sturm \cite{st1},
 by applying the Laplacian comparison principle established by Gigli \cite{gi12} (see Lemma \ref{lap-comp} below)
and the parabolic Harnack inequality in \cite{gm14,jia14} as our main tools.

 In what follows, let $d_{x_0}(x):=d(x_0,x)$ be the distance function for each $x_0$.
\begin{lem}\label{lap-comp}
Let $(X,d,\mu)$ be a $RCD^\ast(K,N)$ space with $K\le 0$ and $N\in [1,\infty)$.
Then the distance function $d_{x_0}\in \mathcal D_\loc(\Delta,\,X\setminus{x_0})$, and
$$\Delta d_{x_0}|_{X\setminus{x_0}}\le \frac{N\tau_{K,N}(d_{x_0})-1}{d_{x_0}}.$$
Above, $\tau_{K,N}\equiv 1$ if $K=0$, and
$\tau_{K,N}(\theta)=\theta\sqrt{ {-K}/N}{\mathrm{cotanh}}\lf(\theta \sqrt{{-K}/N}\r) $ for $\theta\in [0,\infty)$.
\end{lem}

The following parabolic Harnack inequalities are established in \cite{gm14,jia14}.
\begin{lem}[Parabolic Harnack Inequality]\label{harnack-old}
Let $(X,d,\mu)$ be a $RCD^\ast(K,N)$ space with $K\in \rr$ and $N\in [1,\infty)$.
Then for each $0\le f\in \cup_{1\le q<\infty}L^q(X)$, all $0<s<t<\infty$ and $x,y\in X$, it holds that

{\rm (i)} if $K=0$,
$$H_sf(x)\le  H_tf(y)\exp\lf\{\frac{d(x,y)^2}{4(t-s)}\r\}\lf(\frac t s\r)^{N/2};$$

{\rm (ii)} if $K<0$,
$$H_sf(x)\le H_tf(y)\exp\lf\{\frac{d(x,y)^2}{4(t-s)e^{2Kt/3}}\r\}\lf(\frac{1-e^{2Kt/3}}{1-e^{2Ks/3}}\r)^{N/2}.$$
\end{lem}

Now we begin to prove Theorem \ref{lowerbound-nonnegative}.
\begin{proof}[Proof of Theorem \ref{lowerbound-nonnegative}]
(i) The inequality
$$p_t(x,y)\le \frac{C_1(\epsilon)}{\mu(B(y,\sqrt t))}\cdot\exp\Big(-\frac{d^2(x,y)}{(4+\epsilon)t}\Big)$$
follows from Sturm \cite[Corollary 2.5]{st2}, by using the doubling property in Lemma \ref{doubling} and the $L^2$ Poincar\'e
inequality in Lemma \ref{localPoicare}.

To prove the reverse inequality, we set
$N_1:=\min\{m\in\mathbb N|\ m\gs N\}$.
 Since $N\ls N_1$, the space $(X,d,\mu)$ satisfies also  $RCD^*(0,N_1)$. By applying the
 same argument in \cite[p158-p159]{st1} with Lemma \ref{lap-comp}, we conclude that  there exists a constant $C(N_1)$ such that
\begin{equation}\label{3.1}
\int_{B(y,\sqrt t)}p_t(x,z)d\mu(z)
\gs C(N_1)\cdot \exp\Big(-\frac{d^2(x,y)}{4(1-\epsilon)t}-\frac{1+\epsilon^{-1}}{2}\Big),
\end{equation}
for all $t>0$ and all $x,y\in X.$
 Indeed, the argument only used the Laplacian comparison principle for $X$ in Lemma \ref{lap-comp} and an explicit calculation for the heat kernel on Euclidean space $\mathbb R^{N_1}$ of dimension $N_1$ (see  \cite[(3.1)]{st1}).

 Fix any $\epsilon>0$. According to Lemma \ref{harnack-old} (i), we have
\begin{equation*}
\begin{split}
p_{(1+\epsilon)t}(x,y)&\ \gs \mu^{-1}\big(B(y,\sqrt t)\big)\cdot \int_{B(y,\sqrt t)}p_t(x,z)d\mu(z)\cdot \exp(-\frac{1}{4\epsilon})\cdot(1+\epsilon)^{-N_1/2}\\
&\overset{(3.1)}{\gs} C(\epsilon,N_1)\cdot  \mu^{-1}\big(B(y,\sqrt t)\big)\cdot\exp\Big(-\frac{d^2(x,y)}{4(1-\epsilon)t}\Big),
\end{split}
\end{equation*}
for all $t>0$ and all $x,y\in X.$ This implies the desired estimate.
\end{proof}

%
%
%

\begin{proof}[Proof of Corollary \ref{gradient-bound-nonnegative}]
It follows from the Li--Yau inequality
$$|\nabla \log p_t(x,\cdot)|^2(y)-\frac{\partial}{\partial t}\log p_t(x,y)\le \frac N {2t},\quad\hbox{for }\mu\hbox{-a.e. }x,y\in X,$$
in \cite{gm14,jia14} and  Theorem \ref{lowerbound-nonnegative} that, for each $\ez>0$,
\begin{eqnarray*}
|\nabla p_t(x,\cdot)|^2(y)&&\le \frac{N}{2t}p_t(x,y)^2+p_t(x,y)\lf|\Delta p_t(x,\cdot)(y)\r|\\
&&\le  \frac{C_1(\epsilon)}{t\mu(B(y,\sqrt t))^2}\cdot\exp\Big(-2\frac{d^2(x,y)}{(4+\epsilon)t}\Big),
\end{eqnarray*}
as desired.
\end{proof}

We now turn to prove Theorem \ref{lb-negative}. To this end, in particular, we derive a parabolic Harnack inequality for the heat kernel $p_t$ from Lemma \ref{harnack-old} (ii).
\begin{lem}\label{harnack}
Let $(X,d,\mu)$ be a $RCD^*(K,N)$ space with $K<0$ and $N\in [1,\infty)$.
For any $0<s<s+1\le t<\infty$ and $x,y,z\in X$, it holds
\begin{eqnarray*}
p_{t}(x,y)\ge p_{s}(x,z) \exp\lf\{-\frac{d(y,z)^2}{2e^{2K/3}}\r\}\lf(\frac{1-e^{K/3}}{1-e^{2K/3}}\r)^{N/2}\lf(\frac{1-e^{2Ks/3}}{1-e^{2K(t-1/2)/3}}\r)^{N/2}.
\end{eqnarray*}
\end{lem}
\begin{proof}
Write $p_t(x,y)=H_1(p_{t-1}(x,\cdot))(y)$. Then it follows, from Lemma \ref{harnack-old} (ii) that
\begin{eqnarray*}
p_{t}(x,y)&&=H_1(p_{t-1}(x,\cdot))(y)\\
&&\ge H_{1/2}(p_{t-1}(x,\cdot))(z)\exp\lf\{-\frac{d(y,z)^2}{2e^{2K/3}}\r\}\lf(\frac{1-e^{K/3}}{1-e^{2K/3}}\r)^{N/2}\\
&&=p_{t-1/2}(x,z)\exp\lf\{-\frac{d(y,z)^2}{2e^{2K/3}}\r\}\lf(\frac{1-e^{K/3}}{1-e^{2K/3}}\r)^{N/2}\\
&&\ge p_{s}(x,z) \exp\lf\{-\frac{d(y,z)^2}{2e^{2K/3}}\r\}\lf(\frac{1-e^{K/3}}{1-e^{2K/3}}\r)^{N/2}\lf(\frac{1-e^{2Ks/3}}{1-e^{2K(t-1/2)/3}}\r)^{N/2},
\end{eqnarray*}
as desired.
\end{proof}

The following lemma is a particular case of  Sturm \cite[Lemma 1.7]{st2}.
\begin{lem}\label{maximum-p-t}
Let $(X,d,\mu)$ be a $RCD^\ast(K,N)$ space with $K\in\rr$ and $N\in [1,\infty)$.
Suppose that $\psi\in W^{1,2}(X)\cap L^\infty(X)$ with $|\nabla \psi|^2\le \gamma^2$,
and $u$ is a solution to the heat equation $\frac{\partial}{\partial_t}u=\Delta u$ on $X\times [0,\infty)$.
Then for all $0\le s<t<\infty$, it holds
$$\|e^{\psi}u(\cdot,t)\|_{L^2(X)}\le e^{\gamma^2(t-s)}\|e^{\psi}u(\cdot,s)\|_{L^2(X)}.$$
\end{lem}

We can prove Theorem \ref{lb-negative} now. We should mention that the idea of the establishment of the upper bound comes from \cite[Theorem 2.4]{st2} and the lower bound from \cite{st1}.
\begin{proof}[Proof of Theorem \ref{lb-negative}]
{\bf (i) The upper bounds.} Let $\ez\in (0,\frac 12)$, $t>0$ and $x, y\in X$.  By the parabolic Harnack inequality
(Lemma \ref{harnack-old} (ii)), it follows that, for each $z\in B(y,\sqrt t)$, it holds
\begin{equation}\label{harnack1}
p_t(x,y)\le p_{(1+\ez)t}(x,z)\exp\lf\{\frac{d(y,z)^2}{4\ez te^{2K(1+\ez)t/3}}\r\}\lf(\frac{1-e^{2K(1+\ez)t/3}}{1-e^{2Kt/3}}\r)^{N/2},
\end{equation}
and, if $t\geq 1/\ez$, then, by Lemma \ref{harnack},
\begin{eqnarray}\label{harnack2}
p_{t}(x,y)\le p_{(1+\ez) t}(x,z) \exp\lf\{\frac{d(y,z)^2}{2e^{2K/3}}\r\}\lf(\frac{1-e^{2K/3}}{1-e^{K/3}}\r)^{N/2}\lf(\frac{1-e^{2Kt(1+\ez)/3}}{1-e^{2Kt/3}}\r)^{N/2}.
\end{eqnarray}

Let  $\beta\in \rr$, $0\le f\in L^2(X)$ and let $\psi$ be a Lipschitz cut-off function with $|\nabla \psi|^2\le 1$ $\mu$-a.e. in $X$.
Set $u(x,t):=H_t(e^{-\beta \psi}f)$.

Suppose at first that $t< 1/\ez$. By Lemma \ref{harnack-old} (ii), we find that, for all $z\in B(y,\sqrt t)$, it holds
\begin{eqnarray*}
H_{t+t\ez}(e^{-\beta \psi}f)(y)^2\le H_{t+2\ez t}(e^{-\beta \psi}f)(z)^2\exp\lf\{\frac{d(y,z)^2}{4\ez te^{2K(1+2\ez) t/3}}\r\}\lf(\frac{1-e^{2K(1+2\ez)t/3}}{1-e^{2K(1+\ez)t/3}}\r)^{N/2}.
\end{eqnarray*}
This implies that
\begin{eqnarray}\label{2.4}
H_{t+t\ez}(e^{-\beta \psi}f)(y)^2&&\le C(K,\ez)e^{-KN\ez t}\fint_{B(y,\sqrt t)} H_{t+2\ez t}(e^{-\beta \psi}f)(z)^2\,d\mu(z)\\
&&\le C(K,\ez)e^{-KN\ez t}e^{-2\beta\psi(y)+2|\beta|\sqrt t}\fint_{B(y,\sqrt t)}e^{2\beta \psi(z)} H_{t+2\ez t}(e^{-\beta \psi}f)(z)^2\,d\mu(z)\nonumber\\
&&\le C(K,\ez)e^{-KN\ez t}e^{-2\beta\psi(y)+2|\beta|\sqrt t} e^{2\beta^2t}\frac{1}{\mu(B(y,\sqrt t))}\|f\|_{L^2(X)}^2,\nonumber
\end{eqnarray}
where in the last estimate we used Lemma \ref{maximum-p-t}.

On the other hand, since
\begin{eqnarray*}
H_{t+t\ez}(e^{-\beta \psi}f)(y)^2&&=\lf[\int_Xp_{t+t\ez}(y,z)e^{-\beta \psi(z)}f(z)\,d\mu(z)\r]^2\\
&&\ge \lf[\int_{B(x,\sqrt t)}p_{t+t\ez}(y,z)e^{-\beta \psi(z)}f(z)\,d\mu(z)\r]^2\\
&&\ge e^{-2\beta \psi(x)-2|\beta|\sqrt t}\lf[\int_{B(x,\sqrt t)}p_{t+t\ez}(y,z)f(z)\,d\mu(z)\r]^2,
\end{eqnarray*}
which, together with \eqref{2.4}, and taking supremum with respect to $\|f\|_{L^2(B(x,\sqrt t))}\le 1$, yields that
\begin{eqnarray}\label{2.5}
e^{-2\beta \psi(x)-2|\beta|\sqrt t}\|p_{t+t\ez}(y,\cdot)\|_{L^2(B(x,\sqrt t))}^2&&
\le C(K,\ez)e^{-KN\ez t}e^{-2\beta\psi(y)+2|\beta|\sqrt t} e^{2\beta^2t}\mu(B(y,\sqrt t))^{-1}.
\end{eqnarray}
A similar argument as in \eqref{2.4}, by applying the parabolic Harnack inequality \eqref{harnack1}, yields,
\begin{eqnarray}\label{2.6}
p_{t}(x,y)^2&&\le C(K,\ez)\fint_{B(x,\sqrt t)} p_{t+\ez t}(z,y)^2\,d\mu(z)e^{-KN\ez t}\\
&&\le C(K,\ez)\frac{e^{-KN\ez t}}{\mu(B(x,\sqrt t))}\|p_{t+t\ez}(y,\cdot)\|_{L^2(B(x,\sqrt t))}^2.\nonumber
\end{eqnarray}

Let $\beta=\frac{d(x,y)}{2t}$, and choose $\psi$ such that $\psi(x)-\psi(y)$ is sufficiently close to $-d(x,y)$.
Combining the estimates \eqref{2.5} and \eqref{2.6}, we find
\begin{eqnarray*}
p_{t}(x,y)^2&&\le C(K,\ez)\frac{\exp\lf\{{-2KN\ez t+4\beta \sqrt t+2\beta(\psi(x)-\psi(y))+2\beta^2 t}\r\}}{\mu(B(x,\sqrt t))\mu(B(y,\sqrt t))}\\
&&  \le C(K,\ez)\frac{\exp\lf\{-2KN\ez t+\frac{2d(x,y)}{\sqrt t}-\frac{d(x,y)^2}{2t}\r\}}{\mu(B(x,\sqrt t))\mu(B(y,\sqrt t))}.
\end{eqnarray*}
On the other hand, by the local doubling property (Lemma \ref{doubling} (ii)) and $t<1/\ez$, we have
\begin{eqnarray*}
\mu(B(y,\sqrt t))&&\le \mu(B(x,\sqrt t+d(x,y)))\le C\exp\left\{C(K,N)(d(x,y)+\sqrt t)\right\}\mu(B(x,\sqrt t))\\
&&\le C(\ez)\exp\left\{C(K,N,\ez)\frac{d(x,y)}{\sqrt t}\right\}\mu(B(x,\sqrt t)).
\end{eqnarray*}
Combining the above two estimates and the Young inequality, we conclude that
\begin{eqnarray}\label{2.7}
p_{t}(x,y)&&\le \sqrt{C(K,\ez)}\frac{\exp\lf\{-KN\ez t+\frac{d(x,y)}{\sqrt t}-\frac{d(x,y)^2}{4t}+C(K,N,\ez)\frac{d(x,y)}{\sqrt t}\r\}}{\mu(B(y,\sqrt t))}\nonumber\\
&&\le C_1(\ez)\frac{1}{\mu(B(y,\sqrt t))}\exp\lf\{{C_2(\ez)t-\frac{d(x,y)^2}{(4+\ez)t}}\r\}.
\end{eqnarray}

For the case $t\ge 1/\ez$, by using the parabolic Harnack inequality \eqref{harnack2} instead of
\eqref{harnack1} in the proof of \eqref{2.6}, we can conclude that \eqref{2.7} also holds.
This completes the proof of upper bounds.

{\bf (ii) The lower bounds}.
Set $N_1:=\min\{m\in\mathbb N|\ m\gs N\}$. Since $N\ls N_1$, the metric measure space $(X,d,\mu)$ satisfies also  $RCD^*(K,N_1)$.
By applying the Laplacian comparison theorem (Lemma \ref{lap-comp}), and the parabolic maximum principle,
we conclude that, there exists a constant $C(N_1)$ such that
\begin{equation}
\int_{B(y,\sqrt t)}p_t(x,z)d\mu(z)\ge C(K,N_1,\ez)\cdot \exp\Big(-\frac{d^2(x,y)}{4(1-\epsilon)t}-C_2(\epsilon)t\Big)
\end{equation}
for all $t>0$ and all $x,y\in X.$  Notice that, the argument only used the Laplacian
comparison theorem (Lemma \ref{lap-comp}), and an explicit calculation for heat kernel on the hyperbolic
 space $E^{k,N_1}$ of dimension $N_1$ and constant sectional curvature $k=\frac{K}{N_1-1}$;
 see \cite[Corollary 3.6]{st1}.

 For  $t< 1/\ez$, by the parabolic Harnack inequality \eqref{harnack1}, we have
\begin{eqnarray}\label{lb-1}
p_t(x,y)&&\ge C(K,\ez)\fint_{B(y, \sqrt t)}p_t(x,z)\,d\mu(z)e^{KN\ez t}\\
&&\ge \frac{1}{C_1(\epsilon)\mu(B(y,\sqrt t))}\exp\Big(-\frac{d^2(x,y)}{(4-\epsilon)t}-C_2(\epsilon)t\Big).\nonumber
\end{eqnarray}

Suppose now $t\geq\max\{1,1/\ez\}$. By the parabolic Harnack inequality \eqref{harnack2}, we conclude that
\begin{eqnarray*}
p_{(1+\ez)t}(x,y)&&\ge C(K,N)\fint_{B(y,\sqrt t)}p_{t}(x,z) \exp\lf\{-\frac{d(y,z)^2}{2e^{2K/3}}\r\}\,d\mu(z)\lf(\frac{1-e^{2Kt/3}}
{1-e^{2K((1+\ez)t)/3}}\r)^{N/2}\\
&&\ge C(K,N,\ez)\fint_{B(y,\sqrt t)}p_{t}(x,z) \,d\mu(z) \exp\lf\{-\frac{t}{2e^{2K/3}}-C(\ez)t\r\}\\
&&\ge C(K,N,\ez)\exp\Big(-\frac{d^2(x,y)}{4(1-\epsilon)t}-C_2(\ez)t\Big).
\end{eqnarray*}
The proof is completed.
\end{proof}

\begin{proof}[Proof of Corollary \ref{gradient-bound-negative}]

Notice that by \cite[Theorem 1.2]{jia14}, it holds for each $f\in L^1(X)$ and each $t>0$ that,
$$|\nabla \log H_tf|^2\le e^{-2Kt/3}\frac{\Delta H_tf}{H_tf}+\frac{NK}3\frac{e^{-4Kt/3}}{1-e^{-2Kt/3}},\quad\mu\mbox{-a.e.}.$$
This gives, when $0<t\le 1$, that
\begin{equation}\label{bg-gradient}
|\nabla p_t(x,\cdot)|^2(y)\le C(K) |\Delta p_t(x,\cdot)(y)|p_t(x,y)+C(K)t^{-1} p_t(x,y)^2.
\end{equation}
If $t>1$, by writing $p_t(x,y)=H_{1}(p_{t-1}(x,\cdot))(y)$, we can also conclude that
\begin{equation*}
|\nabla p_t(x,\cdot)|^2(y)=|\nabla H_{1}(p_{t-1}(x,\cdot))|^2(y)\le C(K) |\Delta p_t(x,\cdot)(y)|p_t(x,y)+C(K) p_t(x,y)^2,
\end{equation*}

Notice that, by using Davies \cite[Theorem 4]{da97} and Theorem \ref{lb-negative}, we see that for each $t>0,$ and almost all  $x,y\in X$,
\begin{equation}\label{time-gradient}
\lf|\frac{\partial }{\partial t}p_t(x,y)\r|\le \frac{C(\ez)}{t\mu(B(y,\sqrt t))}\exp\lf\{{C_2(\ez)t-\frac{d(x,y)^2}{(4+\ez)t}}\r\}.
\end{equation}
This, together with Theorem \ref{lb-negative} again, yields
\begin{equation*}
|\nabla p_t(x,\cdot)|(y) \le \frac{C(\ez)}{\sqrt t\mu(B(y,\sqrt t))}\exp\lf\{{C_2(\ez)t-\frac{d(x,y)^2}{(4+\ez)t}}\r\},
\end{equation*}
as desired.
\end{proof}

 Gong and Wang \cite{GongWang2001} established a characterization of compactness of Riemannian manifolds by using heat kernel bounds.
 Their arguments work also in our $RCD^\ast(K,N)$ setting. We omit the details of the proof here, but just mention that the heat flow satisfies the semi-group Poincar\'{e} inequality (see \cite[Corollary 2.3]{ags3}): for every $f\in W^{1,2}(X)$ and every $t>0$,
 $$H_t(f^2)-(H_tf)^2\leq \frac{1-e^{-2Kt}}{K}H_t(|\nabla f|^2),\quad\mu\mbox{-a.e. in }X,$$
 and the mass preserving property (see \cite[Section 4]{ags2}): for every $t>0$ and every $f\in L^1(X)\cap L^2(X)$,
 $$\int_X H_tf\,d\mu=\int_X f\,d\mu,$$
{ and for every ball $B(x,r)\subset X$, the bottom of the spectrum of the operator $-\Delta$ in $L^2(B(x,r))$ is bounded above by a positive constant $C(K,N,R)$ with $R> r$, independent of $x$, by applying Lemma \ref{doubling} (see \cite[Proposition 2.1]{st3}).}
\begin{thm}\label{compactness}
Let $(X,d,\mu)$ be a $RCD^\ast(K,N)$ space with $K\in \rr$ and $N\in [1,\infty)$.
Then the following conditions are equivalent.

{\rm (i)} $(X,d)$ is compact;

{\rm (ii)} There exist $x\in X$ and $t_0>0$ such that $\int_{X}p_{t_0}(x,y)^{-1}\,d\mu(y)<\infty;$

{\rm (iii)} There exists $t_0>0$ such that
$\int_{X}p_{t_0}(x,x)\,d\mu(x)<\infty.$
\end{thm}

\section{Stability of solutions to the heat equation}
\hskip\parindent In this section, we apply the heat kernel bounds (Theorem \ref{lowerbound-nonnegative}) to the study
of large time behaviors of the heat kernel and stability of solutions to the heat equation. Our arguments will be based on the method of Li \cite{li86} with  some necessary modifications,
due to lacking of the Stokes' formula (or the Gauss--Green formula).

\begin{defn}[Boundary measure]
For a fixed $x_0\in X$ and $r\in [0,\infty)$, define
$$s(x_0,r):=\limsup_{\delta\to0^+}\frac 1\delta \mu\big(B(x_0,r+\delta)\setminus B(x_0,r)\big).$$
\end{defn}

We remark here that, in the Riemannian manifold $M$ with $\mu$ being the volume measure, it is immediate to see that $s(x_0,r)$ is equal to the $(n-1)$-dimensional Hausdorff measure of $\partial B(x_0,r)$, for every $B(x_0,r)$ in $M$.

The first part in the next lemma is known (see e.g. \cite[Theorem 2.3]{stm5}), and the second part is immediate from the last definition and the local Lipschitz continuity of the function $r\mapsto\mu(B(x_0,r))$ in $(0,\infty)$, for each $x_0\in X$ (see e.g. \cite[p.148]{stm5}).
\begin{lem}\label{volume-growth}
Let $(X,d,\mu)$ be a $RCD^\ast(0,N)$ spaces with $N\in (1,\infty)$. Then for all $x_0\in X$ and $0<r<R<\infty$, it holds
$$\frac{s(x_0,R)}{s(x_0,r)}\le \lf(\frac Rr\r)^{N-1},$$
and
$$\mu(B(x_0,R))=\int_{0}^Rs(x_0,r)\,dr.$$
\end{lem}

\begin{lem}\label{lowerboundary}
Let $(X,d,\mu)$ be a $RCD^\ast(0,N)$ spaces with $N\in (1,\infty)$. If there exists $x_0\in X$, and  $\theta\in (0,\infty)$
such that
$$\liminf_{R\to\infty}\frac{\mu(B(x_0,R))}{R^N}=\theta.$$
Then for each $R>0$, it holds
$$s(x_0,R)\ge N\theta R^{N-1}.$$
\end{lem}
\begin{proof}
From Lemma \ref{volume-growth}, it follows the function $R\mapsto\frac{s(x_0,R)}{R^{N-1}}$ is non-increasing on $(0,\infty)$.
Hence, if there exists $R_0\in (0,\infty)$ such that
$s(x_0,R_0)< N\theta R_0^{N-1},$ then for all $R>R_0$,  it holds
$s(x_0,R)< N\theta R^{N-1}.$
This together with Lemma \ref{volume-growth} implies that for each $R>R_0$,
$$\mu(B(x_0,R)\setminus B(x_0,R_0))=\int_{R_0}^R s(x_0,r)\,dr< \theta \lf(R^N-R_0^{N}\r),$$
and hence
$$\liminf_{R\to\infty}\frac{\mu(B(x_0,R))}{R^N}<\liminf_{R\to\infty}\frac{\theta \lf(R^N-R_0^{N}\r)+\mu(B(x_0,R_0))}{R^N}=\theta.$$
This contradicts with the assumption. Therefore, for each $R>0$, we see that
$s(x_0,R)\ge N\theta R^{N-1},$
as desired.
\end{proof}

\begin{defn}[Boundary Integral]
Let $x_0\in X$ and $r\in [0,\infty)$. Suppose $f\in L^\infty_\loc(X)$. Define the integral of $f$ on $\partial B(p,r)$ as
$$|f|_{\partial B(x_0,r)}:=\limsup_{\delta\to0^+}\frac 1\delta \int_{B(x_0,r+\delta)\setminus B(x_0,r)}f(x)\,d\mu(x).$$
\end{defn}

\begin{lem}\label{co-area}
Let $(X,d,\mu)$ be a $RCD^\ast(K,N)$ space with $K\in \rr$ and $N\in (1,\infty)$. Let $x_0\in X$ and $R\in (0,\infty)$.
Then for each $f\in L^\infty_\loc(X)$, it holds
$$\int_{B(x_0,R)}f(x)\,d\mu(x)=\int_0^R |f|_{\partial B(x_0,r)}\,dr.$$
\end{lem}
\begin{proof}
Notice that, for each fixed $x_0\in X$, the function $r\mapsto\mu(B(x_0,r))$ is locally Lipschitz continuous on $(0,\infty)$.
From this,  we conclude that, for each $f\in L^\infty_\loc(X)$,
the function
$$r\mapsto\int_{B(x_0,r)}f\,d\mu$$
is locally Lipschitz continuous on $(0,\infty)$, and hence, the required equality holds.
\end{proof}

\begin{proof}[Proof of Theorem \ref{large-time}]
It is shown in \cite[Theorem 1.4]{jia14} that
\begin{eqnarray}\label{large-lim}
\lim_{t\to\infty} t^{N/2}p_t(x,y)=C(\theta),
\end{eqnarray}
for some constant $C(\theta)\in (0,\infty)$. Let us prove that $\theta C(\theta)=\omega(N)(4\pi)^{-N/2}$.

Following \cite[Proof of Theorem 4.5]{st1}, by applying the Laplacian comparison theorem (Lemma \ref{lap-comp}),
and the parabolic maximum principle, it follows that, for all $x\in X$ and $t,r>0$,
\begin{eqnarray}
\int_{B(x,r)}p_t(x,y)\,d\mu(y)\ge \int_{B(x_r,r)}\frac{1}{(4\pi t)^{N/2}}\exp\lf\{-\frac{|z|^2}{4t}\r\}\,dz,
\end{eqnarray}
where $x_r=(r, 0,\cdots,0)\in \rr^{N}$. This, together with \eqref{large-lim}, implies that
\begin{eqnarray*}
\mu(B(x,r))C(\theta)&&=\lim_{t\to\infty}t^{N/2}\int_{B(x,r)}p_t(x,y)\,d\mu(y)\ge \frac{\omega(N)r^N}{(4\pi)^{N/2}}.
\end{eqnarray*}
Letting $r\to \infty$, we find that
\begin{eqnarray}\label{upp-large}
\theta C(\theta)\ge \frac{\omega(N)}{(4\pi)^{N/2}}.
\end{eqnarray}

Let us prove the upper bound of $\theta C(\theta)$.
For any $\delta>0$ and all $x,y\in X$ , the parabolic Harnack inequality in Lemma \ref{harnack-old} (i) yields
$$p_t(x,x)^2\le p_{(1+\delta)t}(x,y)^2\exp\lf\{\frac{d(x,y)^2}{2\delta t}\r\}\lf(1+\delta\r)^{N},$$
which implies that, for all $r,\ez>0$,
\begin{eqnarray*}
&&p_t(x,x)^2\frac{\mu(B(x,r+\ez)\setminus B(x,r))}{\ez}\\
&&\quad\le  \exp\lf\{\frac{(r+\ez)^2}{2\delta t}\r\}\lf(1+\delta\r)^{N}
\frac 1\ez\int_{\mu(B(x,r+\ez)\setminus B(x,r))}p_{(1+\delta)t}(x,y)^2\,d\mu(y).\nonumber
\end{eqnarray*}
Letting $\ez\to 0$ and applying Lemma \ref{lowerboundary}, we conclude that for a.e. $r>0$,
\begin{eqnarray}\label{thm4.1-1}
\lf(1+\delta\r)^{N}|p_{(1+\delta)t}(x,\cdot)^2|_{\partial B(x,r)}&&\ge \exp\lf\{-\frac{r^2}{2\delta t}\r\}{s(x,r)}p_t(x,x)^2\cr
&&\ge N\theta r^{N-1}\exp\lf\{-\frac{r^2}{2\delta t}\r\}p_t(x,x)^2.
\end{eqnarray}
Integrating over $(0,\infty)$, and applying Lemma \ref{volume-growth} and Lemma \ref{co-area}, we find that
\begin{eqnarray*}
\lf(1+\delta\r)^{N}t^{N/2}p_{2(1+\delta)t}(x,x)&&=\lf(1+\delta\r)^{N}t^{N/2}\int_X p_{(1+\delta)t}(x,y)^2\,d\mu(y)\\
&&=\lf(1+\delta\r)^{N}t^{N/2}\int_{(0,\infty)} |p_{(1+\delta)t}(x,\cdot)^2|_{\partial B(x,r)}\,dr\\
&&\ge t^{N/2}N\theta p_t(x,x)^2\int_0^\infty r^{N-1}\exp\left(-\frac{r^2}{2\delta t}\right)\,d r\\
&&=\frac{\theta t^{N}}{\omega(N)}(2\pi\delta)^{N/2}p_t(x,x)^2,
\end{eqnarray*}
where in the last inequality, we used \eqref{thm4.1-1}. With \eqref{large-lim}, letting $t\to \infty$, we derive that
\begin{eqnarray*}
\lf(\frac{1+\delta}{2}\r)^{N/2}C(\theta)&&=\lim_{t\to\infty}\lf(\frac{1+\delta}{2}\r)^{N/2}(2(1+\delta)t)^{N/2}p_{2(1+\delta)t}(x,x)\\
&&\ge \lim_{t\to\infty}\frac{(2\pi\delta)^{N/2}}{\omega(N)}\theta t^{N}p_t(x,x)^2=\frac{(2\pi\delta)^{N/2}}{\omega(N)}\theta C(\theta)^2,
\end{eqnarray*}
and hence,
\begin{eqnarray*}
\theta C(\theta) &&\le \frac{\omega(N)}{(2\pi\delta)^{N/2}}\lf(\frac{1+\delta}{2}\r)^{N/2}.
\end{eqnarray*}
Letting $\delta\to\infty$ and using \eqref{upp-large}, we have
\begin{eqnarray*}
\theta C(\theta) &&= \frac{\omega(N)}{(4\pi)^{N/2}},
\end{eqnarray*}
which implies that
$$\lim_{t\to\infty} \mu(B(x_0,\sqrt t))p_t(x,y)=\lim_{t\to\infty} \frac{\mu(B(x_0,\sqrt t))}{t^{N/2}} t^{N/2}p_t(x,y)=\theta C(\theta)= \frac{\omega(N)}{(4\pi)^{N/2}},$$
and hence, we complete the proof.
\end{proof}

The existence of the large time limit of heat kernels yields the following stability property of solutions
to the heat equation. The proof is a slightly modification of the proof of \cite[Theorem 3]{li86} to our non-smooth context.

\begin{thm}\label{stablity-he}
Let $(X,d,\mu)$ be a $RCD^\ast(0,N)$ space with $N\in \cn$ and $N\ge 2$. Let $x_0\in X$.
Suppose that there exists a constant $\theta\in (0,\infty)$ such that $\liminf_{R\to\infty}\frac{\mu(B(x_0,R))}{R^N}=\theta$.
Then, for each $f\in L^\infty(X)$, and for any $x\in X$, the limit $\lim_{t\to\infty}H_t(f)(x)$
exists, if and only if, the limit $$\lim_{r\to\infty}\fint_{B(x,r)}f(y)\,d\mu(y)$$ exists;
moreover, it holds
$$\lim_{t\to\infty}H_t(f)(x)=\lim_{r\to\infty}\fint_{B(x,r)}f(y)\,d\mu(y).$$
\end{thm}
\begin{proof}
It follows from Theorem \ref{large-time} that for each $x\in X$,
$$\lim_{t\to\infty}t^{N/2}p_t(x,x)=\frac{\omega(N)}{\theta (4\pi)^{N/2}}=:C(\theta).$$
This, together with the parabolic Harnack inequality
$$p_s(x,x)\le p_t(x,y)\exp\lf\{\frac{d(x,y)^2}{4(t-s)}\r\}\lf(\frac t s\r)^{N/2},$$
for all $x,y\in X$ and $0<s<t$, yields that there exists $\ez(s)>0$ satisfying $\ez(s)\to 0$ as $s\to \infty$
such that
$$p_t(x,y)\ge (C(\theta)-\ez(s))t^{-N/2}\exp\lf\{-\frac{d(x,y)^2}{4(t-s)}\r\}.$$

For any $0\le g\in L^\infty(X)$ and each $x\in X$, it holds
\begin{eqnarray*}
H_t(g)(x)&&\ge (C(\theta)-\ez(s))t^{-N/2}\int_X \exp\lf\{-\frac{d(x,y)^2}{4(t-s)}\r\}g(y)\,d\mu(y)\\
&&\ge (C(\theta)-\ez(s))t^{-N/2}\int_1^\infty \exp\lf\{-\frac{r^2}{4(t-s)}\r\}|g|_{\partial B(x,r)}\,dr.
\end{eqnarray*}
Let $0\le g\in L^\infty(X)$ and let $z\in \mathbb{R}^N$. {We set for each $x\in X$,
$$\bar{g}_{x}(z):=\frac{|g|_{\partial B(x,|z|)}}{N\omega(N)|z|^{N-1}}\chi_{\{|z|\ge 1\}}(z),\quad z\in \rr^N.$$
By Lemma \ref{volume-growth},  we have
$$0\le \bar{g}_{x}(z)\le \|g\|_{L^\infty(X)}\frac{s(x,|z|)}{N\omega(N)|z|^{N-1}}\chi_{|z|\ge 1}(z)
\le \|g\|_{L^\infty(X)}\frac{s(x,1)}{N\omega(N)}<\infty. $$
Hence, $0\le \bar{g}_x\in L^\infty(\rr^N)$. }

Let $\mathcal{H}_t$ be the heat semigroup on $\mathbb{R}^N$. Then for each $x\in X$, it follows that
\begin{eqnarray}\label{mms-rn}
H_t(g)(x)&&\ge (C(\theta)-\ez(s))\lf(\frac{4\pi(t-s)}{t}\r)^{N/2}\mathcal{H}_{t-s}(\bar{g}_x)(0).
\end{eqnarray}

For a function $f\in L^\infty(X)$, we may assume $|f|\le 1$. By setting $g_1:=1+f$ and $g_2:=1-f$,
we then conclude from \eqref{mms-rn} that
\begin{eqnarray}\label{mms-rn1}
1+H_t(f)(x)&&\ge (C(\theta)-\ez(s))\lf(\frac{4\pi(t-s)}{t}\r)^{N/2}\lf[\mathcal{H}_{t-s}(\bar{1}_x)(0)+\mathcal{H}_{t-s}(\bar{f}_x)(0)\r],
\end{eqnarray}
and
\begin{eqnarray}\label{mms-rn2}
1-H_t(f)(x)&&\ge (C(\theta)-\ez(s))\lf(\frac{4\pi(t-s)}{t}\r)^{N/2}\lf[\mathcal{H}_{t-s}(\bar{1}_x)(0)-\mathcal{H}_{t-s}(\bar{f}_x)(0)\r].
\end{eqnarray}

Using Lemma \ref{lowerboundary} and the definition of $\bar {1}$, we find that
$$\bar{1}_x(z)=\frac{s(x,|z|)}{N\omega(N)|z|^{N-1}}\chi_{\{|z|\ge 1\}}(z)\ge
\frac{\theta}{\omega(N)}\chi_{\{|z|\ge 1\}}(z), \quad \forall \, z\in \rr^N$$
which, together with \eqref{mms-rn1} and \eqref{mms-rn2}, implies that
\begin{eqnarray*}\label{mms-rn3}
&&1-\frac{(C(\theta)-\ez(s))}{C(\theta)}\lf(\frac{t-s}{t}\r)^{N/2} \mathcal{H}_{t-s}(\chi_{\{|z|\ge 1\}})(0) \\
&&\quad\ge (C(\theta)-\ez(s))\lf(\frac{4\pi(t-s)}{t}\r)^{N/2}\mathcal{H}_{t-s}(\bar{f}_x)(0)-H_t(f)(x)\nonumber\\
&&\quad \ge \frac{(C(\theta)-\ez(s))}{C(\theta)}\lf(\frac{t-s}{t}\r)^{N/2} \mathcal{H}_{t-s}(\chi_{\{|z|\ge 1\}})(0)-1.\nonumber
\end{eqnarray*}

Letting $s=\sqrt t$ and $t\to\infty$, we find that
$$\lim_{t\to\infty}\lf[H_t(f)(x)-C(\theta)(4\pi)^{N/2}\mathcal{H}_{t-\sqrt t}(\bar{f}_x)(0)\r]=0,$$
since $\ez(\sqrt t)\to 0$ and $\mathcal{H}_{t-\sqrt t}(\chi_{\{|z|\ge 1\}})(0)\to 1$, as $t\to\infty$.
Hence, $\lim_{t\to\infty}H_t(f)(x)$ exists, if and only if, $\lim_{t\to\infty}\mathcal{H}_{t-\sqrt t}(\bar{f}_x)(0)$
exists. From the proof of \cite[Theorem 3]{li86}, it follows that the limit $\lim_{t\to\infty}\mathcal{H}_{t-\sqrt t}(\bar{f}_x)(0)$
exists, if and only if,
$$\lim_{R\to\infty}\fint_{B(0,R)\subset\rr^N}\bar f_x(z)\,dz$$
exists. Notice that, for $R>1$,
\begin{eqnarray*}
\fint_{B(0,R)\subset\rr^N}\bar f_x(z)\,dz &&=\frac{1}{\omega(N)R^N}\int_1^R |f|_{\partial B(x,r)}\,dr\\
&&=\frac{\mu(B(x,R))}{\omega(N)R^N}\lf[\fint_{B(x,R)}f(y)\,d\mu(y)-\frac{1}{\mu(B(x,R))}\fint_{B(x,1)}f(y)\,d\mu(y)\r].
\end{eqnarray*}
Since $\lim_{R\to \infty}\frac{\mu(B(x,R))}{R^N}=\theta$, we find that,
the limit $\lim_{R\to\infty}\fint_{B(0,R)\subset\rr^N}\bar f_x(z)\,dz$ exists, if and only if,
$\lim_{R\to\infty}\fint_{B(x,R)}f(y)\,d\mu(y)$ exists.

Hence, we see that the limit $\lim_{t\to\infty}H_t(f)(x)$
exists, if and only if, $\lim_{r\to\infty}\fint_{B(x,r)}f(y)\,d\mu(y)$ exists.
Similar arguments as in the proof of \cite[Theorem 3]{li86} imply that the two limits must equal, if they exist.
The proof is therefore completed.
\end{proof}

\section{Riesz transforms}
\hskip\parindent
In this section, we consider the boundedness of the Riesz transform $|\nabla (-\Delta)^{1/2}|$  and the local
version  $|\nabla (-\Delta+a)^{1/2}|$  for some $a>0$. For simplicity, we only consider the case $\mu(X)=\infty$,
the case $\mu(X)<\infty$ follows from minor modifications.

In what follows, we shall let $a_K=0$ if $K=0$, and let $a_K>0$ large enough for $K<0$. Since
$-\Delta+a_K$ is a non-negative self-adjoint operator, we can define fractional powers of
$-\Delta+a_K$ as
$$(-\Delta+a_K)^{-1/2}=\frac{\sqrt{\pi}}{2}\int_0^\infty e^{-a_Ks}e^{s\Delta}\frac{\,ds}{\sqrt s},$$
and
$$(-\Delta+a_K)^{1/2}=\frac{\sqrt{\pi}}{2}\int_0^\infty e^{-a_Ks}(\Delta+a_K) e^{s\Delta}\frac{\,ds}{\sqrt s}.$$
We refer the readers to Yosida \cite[Chapter 9.11]{yo} for more details.

In smooth settings, it is easy to see that $(-\Delta+a_K)^{-1/2}f$ and $(-\Delta+a_K)^{-1/2}f$
 are well defined if $f$ is a smooth function with compact support; see, e.g., \cite{acdh04,cd99}. In the
 non-smooth cases, we need a bit work to show that these two operators are well-defined on some dense spaces.
\begin{defn}[Acting Class]
The acting class $\mathbb{V}(X)$ of the Riesz transform $|\nabla(-\Delta+a_K)^{-1/2}|$ is defined as
$$\mathbb{V}(X):=\{f\in L^1(X)\cap L^\infty(X):\, (-\Delta+a_K)^{-1/2}f\in \mathcal{D}(\Delta)\}.$$
\end{defn}

\begin{rem}
If  $\mu(X)<\infty$, we need to consider the subclass $L^p_0(X)$ of $L^p(X)$,
where the elements $f$ are required to satisfy $f\in L^p(X)$ and $\int_Xf\,d\mu=0$. In this case, we
need to redefine  the acting class $\mathbb{V}(X)$ as
$$\mathbb{V}(X):=\lf\{f\in L^\infty(X):\, \int_Xf\,d\mu=0,\,  (-\Delta+a_K)^{-1/2}f\in \mathcal{D}(\Delta)\r\}.$$
\end{rem}

\begin{lem}\label{dense-class}
Let $(X,d,\mu)$ be a $RCD^\ast(K,N)$ space with $K\in \rr$ and $N\in [1,\infty)$.
For each $p\in [1,\infty)$, $\mathbb{V}(X)$ is dense in $L^p(X)$.
\end{lem}

To prove this lemma, let us recall the following mapping property established in \cite{jia14}.
\begin{prop}\label{map-hk}
Let $(X,d,\mu)$ be  a $RCD^\ast(K,N)$ space, where $K\in \rr$ and $N\in [1,\infty)$.

{\rm (i)} For each $t>0$ and $p\in [1,\infty]$, the operator $H_t$ is bounded on $L^p(X)$ with $\|H_t\|_{p,p}\le 1$.

{\rm (ii)} If $K\ge 0$, then, for each $t>0$,
the operators $\sqrt t|\nabla H_t|$ and $t\Delta H_t$ are bounded on $L^p(X)$ for all $p\in [1,\infty]$.
Moreover, there exists $C>0$, such that, for all $t>0$ and  all $p\in [1,\infty]$,
$$\max\lf\{ \|\sqrt t|\nabla H_t|\|_{p,p}, \|t\Delta H_t\|_{p,p}\r\}\le C.$$

{\rm (iii)} If $K< 0$, then, for each $t>0$,
the operators $\sqrt t|\nabla H_t|$ and $t\Delta H_t$ are bounded on $L^p(X)$ for all $p\in [1,\infty]$.
Moreover, there exists $C>0$, such that, for all $t>0$ and all $p\in [1,\infty]$,
$$\max\lf\{ \|\sqrt {(t\wedge 1)}|\nabla H_t|\|_{p,p}, \|(t\wedge 1)\Delta H_t\|_{p,p}\r\}\le C.$$
\end{prop}

We can now prove Lemma \ref{dense-class}.
\begin{proof}[Proof of Lemma \ref{dense-class}] We only need to show that, for each $f\in L^1(X)\cap L^\infty(X)$, there exists
$\{f_k\}_{k\in \cn}\subset \mathbb{V}(X)$ such that $f_k\to f$ in $L^p(X)$ as $k\rightarrow\infty$.

For each $\ez>0$, define
$$f_\ez:=e^{-a_K\ez}H_{\ez}(f)-e^{-a_K/\ez}H_{1/\ez}(f).$$
From the property of the heat semi-group, it follows that $f_\ez\in L^1(X)\cap L^\infty(X)$, and
$e^{-a_K\ez}H_{\ez}(f)\to f$ in $L^p(X)$ as $\ez\to 0$. From the heat kernel bounds in Theorems \ref{lowerbound-nonnegative}
and \ref{lb-negative}, it follows that for each $x\in X$,
\begin{eqnarray*}
|e^{-a_K/\ez}H_{1/\ez}(f)|(x)&&\le \frac{C_1e^{-a_K/\ez}}{\mu(B(x,\sqrt {1/\ez}))} \int_X  \exp\Big(-\frac{d^2(x,y)}{5/\ez}+C_2/ \ez\Big)|f(y)|\,d\mu(y)\\
&&\le  \frac{C_1e^{-a_K/\ez+C_2/\ez}}{\mu(B(x,\sqrt {1/\ez}))} \|f\|_{L^1(X)},
\end{eqnarray*}
which tends to 0 as $\ez\to 0$, where we choose  $a_K>C_2>0$ if $K<0$. This implies that
$e^{-a_K/\ez}H_{1/\ez}(f)\to 0$ in $L^p(X)$, and hence, $f_\ez\to f$ in $L^p(X)$, as $\ez\to 0$.

Let us show that $f_\ez\in \mathbb{V}(X)$. By the analytic property of the heat semi-group, and the fact
\begin{eqnarray*}
(-\Delta+a_K)^{-1/2}f_\ez= \int_\ez^{1/\ez}\sqrt{s}(-\Delta+a_K)^{1/2} e^{s(\Delta-a_K)}(f)\frac{\,ds}{\sqrt s},
\end{eqnarray*}
we conclude that $$\| (-\Delta+a_K)^{-1/2}f_\ez\|_{L^2(X)}\le C(\ez)\|f\|_{L^2(X)},$$
for some constant $C(\epsilon)>0$.

To show that $(-\Delta+a_K)^{-1/2}f_\ez\in \mathcal {D}(\Delta)$, we write
\begin{eqnarray*}
(-\Delta+a_K)^{-1/2}f_\ez=\frac{\sqrt{\pi}}{2}\int_0^\infty \int_\ez^{1/\ez}(-\Delta+a_K) e^{(s+t)(\Delta-a_K)}(f)\,dt\frac{\,ds}{\sqrt s}.
\end{eqnarray*}

If $K=0$, then by using Theorem \ref{lowerbound-nonnegative}, the Minkowski inequality and Proposition
\ref{map-hk}, we find
\begin{eqnarray*}
\|\nabla (-\Delta)^{-1/2}f_\ez\|_{L^2(X)}&&\le \frac{\sqrt{\pi}}{2}\int_0^\infty \int_\ez^{1/\ez}\|\nabla (-\Delta)e^{(s+t)\Delta}(f)\|_{L^2(X)}\,dt\frac{\,ds}{\sqrt s}\\
&&\le C \int_0^\infty \int_\ez^{1/\ez}\|\nabla e^{\frac{(s+t)}2\Delta} (-\Delta)e^{\frac{(s+t)}{2}\Delta}(f)\|_{L^2(X)}\,dt\frac{\,ds}{\sqrt s} \\
&&\le  C\int_0^\infty \int_\ez^{1/\ez}\frac{1}{(s+t)^{3/2}}\|f\|_{L^2(X)}\,dt\frac{\,ds}{\sqrt s}\\
&&\le C(\ez)\|f\|_{L^2(X)},
\end{eqnarray*}
and
\begin{eqnarray*}
\|\Delta (-\Delta)^{-1/2}f_\ez\|_{L^2(X)}&&\le \frac{\sqrt{\pi}}{2}\int_0^\infty \int_\ez^{1/\ez}\|\Delta (-\Delta)e^{(s+t)\Delta}(f)\|_{L^2(X)}\,dt\frac{\,ds}{\sqrt s}\\
&&\le  \frac{\sqrt{\pi}}{2}\int_0^\infty \int_\ez^{1/\ez}\frac{1}{(s+t)^{2}}\|f\|_{L^2(X)}\,dt\frac{\,ds}{\sqrt s}\\
&&\le C(\ez)\|f\|_{L^2(X)},
\end{eqnarray*}
for some constant $C(\epsilon)>0$. These show that $f_\ez\in \mathbb{V}(X)$ when $K=0$.

Suppose now $K<0$. Using Theorem \ref{lb-negative}, the Minkowski inequality and Proposition
\ref{map-hk}, we find
\begin{eqnarray*}
\|\nabla (-\Delta+a_K)^{-1/2}f_\ez\|_{L^2(X)}&&\le \frac{\sqrt{\pi}}{2}\int_0^\infty \int_\ez^{1/\ez}\|\nabla (-\Delta+a_K)e^{(s+t)(\Delta-a_K)}(f)\|_{L^2(X)}\,dt\frac{\,ds}{\sqrt s}\\
&&\le  \frac{\sqrt{\pi}}{2}\int_0^\infty \int_\ez^{1/\ez}\frac{e^{-a_K(t+s)}}{[(s+t)\wedge 1]^{3/2}}\|f\|_{L^2(X)}\,dt\frac{\,ds}{\sqrt s}\\
&&\le C(\ez,a_K)\|f\|_{L^2(X)},
\end{eqnarray*}
and
\begin{eqnarray*}
\|\Delta (-\Delta+a_K)^{-1/2}f_\ez\|_{L^2(X)}&&\le \frac{\sqrt{\pi}}{2}\int_0^\infty \int_\ez^{1/\ez}\|\Delta (-\Delta+a_K)e^{(s+t)(\Delta-a_K)}(f)\|_{L^2(X)}\,dt\frac{\,ds}{\sqrt s}\\
&&\le  \frac{\sqrt{\pi}}{2}\int_0^\infty \int_\ez^{1/\ez}\frac{e^{-a_K(t+s)}}{[(s+t)\wedge 1]^{2}}\|f\|_{L^2(X)}\,dt\frac{\,ds}{\sqrt s}\\
&&\le C(\ez,a_K)\|f\|_{L^2(X)},
\end{eqnarray*}
for some constant $C(\ez,a_K)>0$. Hence, we can conclude now that $ f_\ez\in \mathbb{V}(X)$ for each $\ez>0$.

Therefore, $\mathbb{V}(X)$ is dense in $L^p(X)$ for each $p\in [1,\infty)$. The proof is completed.
\end{proof}

Now we present the main results of this section. The first theorem is on the $L^p$ boundedness of the Riesz transform $|\nabla (-\Delta)^{-1/2}|$. We recall that a sub-linear operator $T$ is of weak type $(1,1)$ if, for any $f\in L^1(X)$, there exists a constant $C>0$ such that $$\sup_{\lambda>0}\lambda\mu(\{x\in X:\, |Tf(x)|>\lambda\})\leq C\|f\|_{L^1(X)}.$$
\begin{thm}\label{Riesz}
Let $(X,d,\mu)$ be a $RCD^\ast(0,N)$ space with $N\in [1,\infty)$.
Then, the Riesz transform $|\nabla (-\Delta)^{-1/2}|$, initially defined on $\mathbb{V}(X)$, extends to a
sub-linear operator of weak type $(1,1)$ and it is also bounded on $L^p(X)$, for each $p\in (1,\infty)$.
\end{thm}

The next one is the local version of Theorem \ref{Riesz}.
\begin{thm}\label{localRiesz}
Let $(X,d,\mu)$ be a $RCD^\ast(K,N)$ space with $K<0$ and $N\in [1,\infty)$.
Then, for large enough $a_K>0$, the local Riesz transform $|\nabla (-\Delta+a_K)^{-1/2}|$,
initially defined on $\mathbb{V}(X)$,  extends to a  bounded sub-linear operator on $L^p(X)$, for each $p\in (1,\infty)$.
\end{thm}

Notice that, as pointed out by \cite[p.6]{BCF14}, although the results from \cite{acdh04} were stated in manifolds, their method of proofs indeeds work on metric measure spaces. Therefore, Theorem \ref{Riesz} and Theorem \ref{localRiesz} follow directly from \cite[Theorem 1.4]{acdh04} and \cite[Theorem 1.7]{acdh04}
by using Corollary \ref{gradient-bound-nonnegative} and Corollary \ref{gradient-bound-negative}, respectively.

In what follows, we outline two different proofs of these results, since the proofs are basically identical to \cite{cd99,acdh04}.

\subsection{Proofs by using heat kernel gradient estimates}
\hskip\parindent
Now we begin the proof of the main results by combing the heat kernel gradient estimate in Corollaries \ref{gradient-bound-nonnegative} and  \ref{gradient-bound-negative}, and applying the original method of \cite{cd99} and \cite{acdh04}.
\begin{proof}[Proof of Theorem \ref{Riesz}]
For each $f\in \mathbb{V}(X)$, it follows from the definition of $\mathbb{V}(X)$ that,
$$\|\nabla (-\Delta)^{-1/2}(f)\|_{L^2(X)}=\lf(\int_X (-\Delta)^{-1/2}(f)(x) (-\Delta)^{1/2}(f)(x)\,d\mu(x)\r)^{1/2}=\|f\|_{L^2(X)}.$$
Then by Lemma \ref{dense-class}, we conclude that $|\nabla (-\Delta)^{-1/2}|$ is bounded on $L^2(X)$. Following the method in \cite{cd99} and \cite{acdh04}, we divide the proof into two cases, namely,  $p\in (1,2)$ and $p\in (2,\infty)$.

Following the proofs in \cite[Sections 2 and 3]{cd99}, by applying the Calder\'{o}n--Zygmund decomposition
(see, e.g., \cite[pp.1154-1155]{cd99}) for functions and the Marcinckiewicz interpolation theorem, in order to show the boundedness of $|\nabla (-\Delta)^{-1/2}|$ on $L^p(X)$ for all $p\in(1,2)$, we only need to prove that $|\nabla (-\Delta)^{-1/2}|$ is of weak type $(1,1)$.  It will follow from by using the heat kernel estimate in Theorem \ref{lowerbound-nonnegative} and by showing that, there exist constants $C,c>0$ such that for all $s,t>0$ and $\mu$-a.e. $x\in X$,
$$\int_{d(x,y)\ge t^{1/2}}|\nabla p_s(x,\cdot)|(y)\,d\mu(y)\le Ce^{-c t/s}s^{-1/2}.$$
Indeed, by using the gradient estimate \eqref{kernelgradient} and the doubling property in Lemma \ref{doubling} (i), we see that
\begin{eqnarray*}
\int_{d(x,y)\ge t^{1/2}}|\nabla  p_s(x,\cdot)|(y)\,d\mu(y)&&\le \sum_{k=1}^\infty \frac{C\mu\big(B(x,2^{k}\sqrt t)\setminus B(x,2^{k-1}\sqrt t)\big)}{\sqrt s\mu(B(x,\sqrt s))}\exp\lf\{-\frac{2^{2k}t}{20s}\r\} \\
&&\le C \exp\lf\{-\frac{ t}{10s}\r\}s^{-1/2} .
\end{eqnarray*}
Hence,  $|\nabla (-\Delta)^{-1/2}|$ is  bounded on $L^p(X)$, for all $1<p<2$.

For the remaining case that $p\in (2,\infty)$, combining the estimate of heat kernel in Theorem \ref{lowerbound-nonnegative} and its gradient estimate \eqref{kernelgradient} and the real variable result on singular integrals in \cite[Theorem 2.1]{acdh04}, and following the proof in
\cite[pp.936-938]{acdh04},  we conclude that $|\nabla (-\Delta)^{-1/2}|$ is bounded on $L^p(X)$ for all $p\in (2,\infty)$, which completes
the proof.
\end{proof}

Now we begin to prove Theorem \ref{localRiesz}.
\begin{proof}[Proof of Theorem \ref{localRiesz}]
 For each $f\in \mathbb{V}(X)$, it follows from the definition of $\mathbb{V}(X)$ that,
\begin{eqnarray*}
&&\|\nabla (-\Delta+a_K)^{-1/2}(f)\|_{L^2(X)}^2+a_K\|(-\Delta+a_K)^{-1/2}(f)\|_{L^2(X)}^2\\
&&=\int_X \langle(-\Delta+a_K)(-\Delta+a_K)^{-1/2}(f),(-\Delta+a_K)^{-1/2}(f)\rangle\,d\mu\\
&&=\|f\|_{L^2(X)}^2.
\end{eqnarray*}
Hence, $|\nabla (-\Delta+a_K)^{-1/2}|$ is bounded on $L^2(X)$ with
\begin{eqnarray*}
&&\|\nabla (-\Delta+a_K)^{-1/2}(f)\|_{L^2(X)}\le \|f\|_{L^2(X)}.
\end{eqnarray*}

The rest of the proof follows from similar proofs in the cases $p\in (1,2)$ and $p\in(2,\infty)$ as in \cite{cd99,acdh04}, by  noticing that
$$|\nabla (-\Delta+a_K)^{-1/2}f|\le \frac{\sqrt \pi}{2}\int_0^\infty |\nabla e^{t(\Delta-a_K)}(f)|\frac{\,dt}{\sqrt t},$$
and combining the estimates of the heat kernel and its gradient in Theorem \ref{lb-negative} and in Corollary \ref{gradient-bound-negative}, respectively. The proof is then finished.
\end{proof}

\subsection{Proofs without explicit heat kernel gradient estimates}
\hskip\parindent
In fact, also by the method in \cite{cd99,acdh04}, the $L^p$ boundedness of the Riesz transform $|\nabla (-\Delta)^{-1/2}|$ and the local Riesz transform $|\nabla (a_K-\Delta)^{-1/2}|$ for every $p\in (1,\infty)$ can be shown without using the explicit heat kernel gradient estimates \eqref{kernelgradient} and \eqref{kernelgradientnegative} .  The approach demands the following main ingredients: (1) doubling property, (2) stochastic completeness, (3) $L^2$ version of the local weak Poincar\'{e} inequality, (4) Caccioppoli type inequalities, (5) $L^p$ boundedness of $\sqrt{t}|\nabla H_t|$, (6) Davies--Gaffney type estimates, and (7) $L^2$ boundedness of the Riesz transform.

Now we are going to establish the missing parts (4)-(6) one by one and then give another proof of Theorems \ref{Riesz} and \ref{localRiesz}.

The Caccioppoli type estimate for the semi-group $H_t$ can be deduced from the following inequality (see \cite[Corollary 2.3]{ags3}),
for every $t>0$ and $f\in L^2(X)$,
\begin{eqnarray}\label{reversedPoincare}
\frac{e^{2Kt}-1}{K}|\nabla H_tf|^2 + \frac{e^{2Kt}-2Kt-1}{NK^2}(\Delta H_tf)^2\leq H_t(f^2)-(H_t f)^2,\quad\mu\mbox{-a.e. in }X.
\end{eqnarray}
Here and in the next proposition, $N=\infty$ is allowable.
\begin{prop}\label{caccioppoli}
Let $(X,d,\mu)$ be a $RCD^\ast(K,N)$ space with $K\in\R$ and $N\in [1,\infty)$ and let $p\in [2,\infty]$. For any $t>0$, it holds
\begin{eqnarray}\label{caccioppoliinequality}
\||\nabla H_tf|\|_{L^p(X)} \leq \sqrt{\frac{K}{e^{2Kt}-1}}\|f\|_{L^p(X)},\quad\mbox{for every } f\in L^2\cap L^p(X);
\end{eqnarray}
in particular, for $K\geq 0$, it holds
\begin{eqnarray}\label{caccioppoliinequality+}
\||\nabla H_tf|\|_{L^p(X)} \leq \frac{1}{\sqrt{2t}}\|f\|_{L^p(X)},\quad\mbox{for every } f\in L^2\cap L^p(X),
\end{eqnarray}
and for $K<0$, it holds
\begin{eqnarray}\label{caccioppoliinequality-}
\||\nabla H_tf|\|_{L^p(X)} \leq \frac{1}{\sqrt{2(t\wedge 1)}}\|f\|_{L^p(X)},\quad\mbox{for every } f\in L^2\cap L^p(X).
\end{eqnarray}
\end{prop}
\begin{proof}
By \eqref{reversedPoincare}, it is immediate to get, for every $t>0$ and $f\in L^2(X)$,
\begin{eqnarray}\label{reversedPoincare-1}
|\nabla H_tf|^2 \leq \frac{K}{e^{2Kt}-1}H_t(f^2),\quad\mu\mbox{-a.e. in }X.
\end{eqnarray}
Note that for $t>0$, $H_t$ is a contraction in all $L^p(X)$ with $p\in [1,\infty]$. For $p=\infty$, the proof of \eqref{caccioppoliinequality} is obvious by \eqref{reversedPoincare-1}. Now take $p\in [2,\infty)$ and $f\in L^2\cap L^p(X)$. Then, by \eqref{reversedPoincare-1}, we have
\begin{eqnarray*}
\||\nabla H_tf|\|_{L^p(X)} \leq \sqrt{\frac{K}{e^{2Kt}-1}}\left(\int_X\big(H_tf^2\big)^{\frac{p}{2}}\, d\mu\right)^{\frac{1}{p}}\leq \sqrt{\frac{K}{e^{2Kt}-1}}\|f\|_{L^p(X)},
\end{eqnarray*}
which is just \eqref{caccioppoliinequality}. For $K\geq 0$, \eqref{caccioppoliinequality+} is immediately implied by \eqref{caccioppoliinequality} and the elementary inequality $e^{2Kt}\geq 1+2Kt$. For $K<0$, \eqref{caccioppoliinequality-} follows from the fact that the function $t\mapsto \frac{e^{2Kt}-1}{K}$ is decreasing in $[0,\infty)$ with maximum value 2 at $t=0$.
\end{proof}

The first Davies--Gaffney type estimate can be established by the same proof of \cite[Lemma 3.6]{BaudoinGarofalo2012} with minor modifications. In fact, if we let
$$\Gamma(f,g)=\la \nabla f, \nabla g\ra,\quad\mbox{for every }f,g\in W^{1,2}(X),$$
then $\Gamma: W^{1,2}(X)\times W^{1,2}(X)\rightarrow L^1(X)$ is a \emph{carr\'{e} du champ}. The second one can be established by the general method since $(H_t)_{t\geq 0}$ is an analytic semi-group in $L^p(X)$ for every $p\in (1,\infty)$; see \cite[Lemma 3.7]{BaudoinGarofalo2012}. We present them in the next lemma and omit the details here.
\begin{lem}\label{gaffney-1}
Let $(X,d,\mu)$ be a $RCD^\ast(K,N)$ space with $K\in\R$ and $N\in [1,\infty)$. For every closed subsets $E,F$ of $X$, every $t>0$ and every function $f\in L^2(X)$ with support in $E$, it holds
\begin{eqnarray}\label{gaffney1}
\|H_tf\|_{L^2(F)} \leq e^{-\frac{d^2(E,F)}{4t}}\|f\|_{L^2(E)},
\end{eqnarray}
and there exists a constant $C>0$, such that
$$\|t\Delta H_tf\|_{L^2(F)}\leq Ce^{-\frac{d^2(E,F)}{6t}} \| f\|_{L^2(E)}.$$
\end{lem}
\begin{rem}\rm
One may also use two side heat kernel bounds, Theorem \ref{lowerbound-nonnegative} and
 Theorem \ref{lb-negative}, and Phragm\'en-Lindel\"of theorem (cf. \cite[Proposition 2.1]{CS}) to show that
$$\|H_tf\|_{L^2(F)} \leq e^{-\frac{d^2(E,F)}{5t}}\|f\|_{L^2(E)}.$$
Notice that in the above estimate the constant in the exponential term is larger than in \eqref{gaffney1}.
\end{rem}

Applying the similar argument as the proof of \eqref{gaffney1}, we can prove the following corollary, which is in fact equivalent to \eqref{gaffney1}.
\begin{cor}\label{gaffneycorollary}
Let $(X,d,\mu)$ be a $RCD^\ast(K,N)$ space with $K\in\R$ and $N\in [1,\infty)$. For every $t>0$ and $f_i\in L^2(B_i)$ with support in two balls $B_i:=B(x_i,r_i)$ in $X$, $i=1,2$, it holds
\begin{eqnarray}\label{gaffney2}
\int_X (H_tf_1)f_2\, d\mu\leq e^{-\frac{d^2(B_1,B_2)}{4t}}\|f_1\|_{L^2(B_1)}\|f_2\|_{L^2(B_2)}.
\end{eqnarray}
\end{cor}

In order to establish the third Davies--Gaffney type estimates, we need to construct the Lipschitz cut-off function with quantitative estimate on its gradient in $X$, which is possible since for fixed $x\in X$, the function $y\mapsto d(x,y)$ is Lipschitz continuous in $X$ with respect to the distance $d$ itself.
\begin{lem}\label{cutoff}
Let $\epsilon>0$. For a closed subset $F$ of $X$, there exists a function $\chi\in \Lip(X)$ such that  $\chi\in [0,1]$, $\chi\equiv 1$ on $F^{(\epsilon/2)}$ and $\chi\equiv 0$ in $X \setminus F^{(\epsilon)}$, and satisfying
$$|\nabla \chi|\leq \frac{2}{\epsilon},\quad\mu\mbox{-a.e. in }X,$$
where $F^{(\epsilon)}:=\{x\in X: d(x,F)<\epsilon\}$ is the $\epsilon$-neighborhood of $F$ with respect to $d$.
\end{lem}
\begin{proof}
Choose
$$\chi(x)=\frac{\big(\frac{\epsilon}{2}-d(x,F^{(\epsilon/2)})\big)^+}{\frac{\epsilon}{2}}\wedge 1,\quad x\in X,$$
where $d(x,A):=\inf_{y\in A} d(x,y)$ for a subset $A$ in $X$. We complete the proof.
\end{proof}

{The third Davies--Gaffney type estimate is presented in the next lemma. The key ingredient in the proof is the following inequality: for every $t>0$,
\begin{eqnarray*}
t\int_X \chi^2 |\nabla H_tf|^2\, d\mu \leq \|t\Delta H_tf\|_{L^2(\bar{F}^{(\epsilon)})}\|H_tf\|_{L^2(X)}+2\Big(t\int_X |\nabla\chi|^2(H_tf)^2\,d\mu \Big)^{\frac{1}{2}}\Big(t\int_X \chi^2|\nabla H_tf|^2\, d\mu \Big)^{\frac{1}{2}},
\end{eqnarray*}
where $\bar{F}^{(\epsilon)}$ is the closure of $F^{(\epsilon)}$ and $\chi$ is the cut-off function constructed in Lemma \ref{cutoff}, and then apply the first two Davies--Gaffney type estimates in Lemma \ref{gaffney-1} (see \cite[Lemma 3.10]{BaudoinGarofalo2012} for the smooth manifold case). We also omit the details here.}
\begin{lem}\label{gaffney-3}
Let $(X,d,\mu)$ be a $RCD^\ast(K,N)$ space with $K\in\R$ and $N\in [1,\infty)$. For every closed sets $E,F \subset X$, every $t>0$ and every function $f\in L^2(X)$ with support in $E$, there exist constants $C>0$ and $\beta>0$ such that
$$\sqrt{t}\||\nabla H_tf|\|_{L^2(F)}\leq C e^{-\beta\frac{d^2(E,F)}{t}} \| f\|_{L^2(E)}.$$
\end{lem}

Now we are in position to present another proof of Theorem \ref{Riesz}.
\begin{proof}[Proof of Theorem \ref{Riesz} for $p\in (1,2)$]
As the proof given in Section 5.1, we only need to show that $|\nabla (-\Delta)^{-1/2}|$ is of weak type $(1,1)$, which is implied by combining Theorem \ref{lowerbound-nonnegative}, the doubling property in Lemma \ref{doubling} (i) and the third Davies--Gaffney type estimate in Lemma \ref{gaffney-3}.
\end{proof}

\begin{rem}
The result that the Riesz transform $|\nabla (-\Delta)^{-1/2}|$ is bounded in $L^p(X)$ for all $p\in (1,2)$ is in fact can also be shown following the proof of \cite[Theorem 5]{Sikora2004}, once we let $\alpha=\frac{1}{2}$, $L=\Delta$ and $A=|\nabla\cdot|$ in the aforementioned theorem and combine with Theorem \ref{lowerbound-nonnegative} and Corollary \ref{gaffneycorollary}.
\end{rem}

\begin{proof}[Proof of Theorem \ref{Riesz} for $p\in (2,\infty)$]
Following the method in \cite[Sections 2 and 3]{acdh04}, we only need to combine Lemma \ref{doubling} (i), the first part of Lemma \ref{localPoicare}, Proposition \ref{caccioppoli}, Lemma \ref{gaffney-1} and Lemma \ref{gaffney-3}, and then complete the proof.
\end{proof}

As for Theorem \ref{localRiesz}, by the localization technique in \cite[Section 5]{DuongRobinson1996} (see also \cite[Section 4]{acdh04} in the Riemannian setting), the local case can be proved as in the global case.

\subsection*{Acknowledgment}
\hskip\parindent  The authors would like to thank the referee for
the very detailed and valuable report, which improved the article.

\noindent Renjin Jiang

\vspace{0.1cm}
\noindent
School of Mathematical Sciences, Beijing Normal University,
Laboratory of Mathematics and Complex Systems, Ministry of Education, 100875, Beijing, China

\vspace{0.2cm}
\noindent{\it E-mail address}: \texttt{rejiang@bnu.edu.cn}

\vspace{0.2cm}
\noindent Huaiqian Li

\vspace{0.1cm}
\noindent
School of Mathematics, Sichuan University, Chengdu 610064,  China

\vspace{0.2cm}
\noindent{\it E-mail address}: \texttt{hqlee@amss.ac.cn}

\vspace{0.2cm}
\noindent Huichun Zhang

\vspace{0.1cm}
\noindent Department of Mathematics, Sun Yat-sen University, Guangzhou 510275, China

\vspace{0.2cm}
\noindent{\it E-mail address}: \texttt{zhanghc3@mail.sysu.edu.cn}

\end{document}